\documentclass[11pt]{amsart}
\usepackage[T1]{fontenc}
\usepackage{lmodern}
\usepackage[protrusion=true,expansion=false]{microtype}
\usepackage{amsmath,amssymb,mathtools}
\usepackage{xcolor}
\usepackage[margin=1.15in]{geometry}
\usepackage[colorlinks=true,linkcolor=blue!45!black,citecolor=blue!45!black,urlcolor=blue!45!black]{hyperref}
\hypersetup{
  pdftitle={The Artin--Hasse p-section: weighted convolutions and p-adic recovery},
  pdfauthor={Ben Clare}
}

\newcommand{\Fp}{\mathbf F_p}
\newcommand{\Zp}{\mathbf Z_p}
\newcommand{\Qp}{\mathbf Q_p}
\newcommand{\AH}{\operatorname{AH}_p}
\newcommand{\ord}{\operatorname{ord}}
\newcommand{\dd}{\mathop{}\!\mathrm d}
\newcommand{\Wcal}{\mathcal W}
\newcommand{\vp}{v_p}
\newtheorem{theorem}{Theorem}[section]
\newtheorem{lemma}[theorem]{Lemma}
\newtheorem{corollary}[theorem]{Corollary}
\newtheorem*{corollary*}{Corollary}
\newtheorem{proposition}[theorem]{Proposition}
\theoremstyle{remark}
\newtheorem{remark}[theorem]{Remark}

\title[The Artin--Hasse $p$-section]{The Artin--Hasse $p$-section: weighted convolutions and $p$-adic recovery}
\author{Ben Clare}
\address{Independent researcher, United Kingdom}
\email{benclare@gmail.com}
\date{7 September 2026}

\subjclass[2020]{Primary 33E50; Secondary 11B68, 11B73, 11B85}
\keywords{Artin--Hasse exponential, Bernoulli numbers, Stirling numbers, Wilson quotient, weighted convolution, automatic sequence, Christol theorem, $p$-adic lifting, $p$-adic contraction, Witt coordinates}

\begin{document}
\begin{abstract}
Let $p$ be an odd prime and let $a_n\in\Fp$ be the reduction modulo $p$ of the $n$th coefficient of the Artin--Hasse exponential.  For the weighted $p$-section convolution
\[
 W_k=\sum_{r=0}^k(-1)^r r\,a_{rp}a_{(k-r)p},
\]
we prove the conjecture of Avitabile and Mattarei for $1<k<p$ and extend it to a global power-series identity determining the entire sequence $(W_k)_{k\geq0}$.  The identity yields an explicit base-$p$ digit formula; in particular, $(W_k)$ is $p$-automatic and admits a finite-state evaluator.  Independently, the exact conjugated $p$-section equation gives a strict $p$-adic fixed-point iteration for its logarithmic derivative.  Big-Witt reconstruction then recovers the $p$-section, and hence all Artin--Hasse coefficients, to arbitrary prescribed $p$-adic precision by a purely modular procedure.
\end{abstract}
\maketitle

\section{Introduction}
Fix an odd prime $p$.  The Artin--Hasse exponential is
\[
 \AH(X)=\exp\!\left(\sum_{i\geq0}\frac{X^{p^i}}{p^i}\right)
       =\sum_{n\geq0}u_nX^n\in\Zp[[X]].
\]
Write $a_n$ for the image of $u_n$ in $\Fp$, and put
\[
 G(X)=\sum_{n\geq0}(-1)^na_{np}X^n,
 \qquad
 W_k=\sum_{r=0}^{k}(-1)^r r\,a_{rp}a_{(k-r)p}.
\]
Avitabile and Mattarei \cite{AM} introduced
\[
 \gamma_p(X)=\sum_{n=1}^{p-2}\frac{B_n}{n}X^{p-n}\in\Fp[X],
\]
where $B_1=-1/2$, and conjectured that
\begin{equation}\label{eq:AMconj}
 W_k=\frac{B_{p-k}}{k}\qquad(1<k<p).
\end{equation}
They showed that this is equivalent to
\begin{equation}\label{eq:old-logform}
 X\frac{G'(X)}{G(X)}\equiv w_pX-\gamma_p(X)\pmod{X^p},
\end{equation}
where
\[
 w_p=\frac{(p-1)!+1}{p}\pmod p
\]
is the Wilson quotient.

The purpose of this paper is to replace the truncation in \eqref{eq:old-logform} by an identity in the full power-series ring.  Define
\begin{equation}\label{eq:qdef}
 q(X)=\sum_{i\geq0}X^{p^i}=X+X^p+X^{p^2}+\cdots\in\mathbf Z[[X]],
\end{equation}
and let
\[
 P_p(Z)=(Z)_p=Z(Z-1)\cdots(Z-p+1).
\]
Since $P_p(Z)\equiv Z^p-Z\pmod p$ and $q(X)^p=q(X)-X$ in $\Fp[[X]]$, the series $P_p(q(X))+X$ is coefficientwise divisible by $p$.  Hence
\begin{equation}\label{eq:Hdef}
 \mathcal H_p(X):=\frac{P_p(q(X))+X}{p}\in\mathbf Z[[X]]
\end{equation}
is well defined.  A bar will denote reduction modulo $p$.

The paper has two branches with a common source.  Proposition~\ref{prop:conjugated}, the exact conjugated $p$-section equation, is the shared engine: Sections~\ref{sec:global-proof}--\ref{sec:automatic} use it to study the weighted convolution, while the recovery argument in Section~\ref{sec:recovery} returns directly to the same equation $p$-adically.  The opening blind-spot proposition of that section uses the convolution identities only as motivation; the fixed-point and Witt-reconstruction proofs themselves do not depend on Theorems~\ref{thm:global}--\ref{thm:automatic}.  This is why the two developments are kept in one paper.

\subsection*{Part I: Weighted convolution and finite-state structure}
Our main result is global.

\begin{theorem}[Global weighted-convolution formula]\label{thm:global}
In $\Fp[[X]]$ one has
\begin{equation}\label{eq:global-log}
 X\frac{G'(X)}{G(X)}=\overline{\mathcal H_p(X)}
 =\overline{\frac{P_p(q(X))+X}{p}},
\end{equation}
and
\begin{equation}\label{eq:global-W}
 \Wcal(X):=\sum_{k\geq0}W_kX^k
 =\frac{X}{q(X)}\,\overline{\mathcal H_p(X)}.
\end{equation}
Equivalently, for every $k\geq0$,
\begin{equation}\label{eq:all-k-coeff}
 W_k=[X^k]\left\{\frac{X}{q(X)}\,
 \overline{\frac{P_p(q(X))+X}{p}}\right\}.
\end{equation}
\end{theorem}

\begin{corollary*}[Avitabile--Mattarei conjecture]
For $k=1$ one has
\[
 W_1=w_p,
\]
and for $1<k<p$,
\[
 W_k=\frac{B_{p-k}}{k}\qquad\text{in }\Fp.
\]
In particular, Theorem~\ref{thm:global} proves the conjecture of Avitabile and Mattarei in its original range.
\end{corollary*}
\begin{proof}
By \eqref{eq:Hdef}, since $q(X)=X+O(X^p)$ and the coefficient of $Z$ in $P_p(Z)$ is $(p-1)!$, the coefficient of $X$ in $\overline{\mathcal H_p(X)}$ is $w_p$.  Since $X/q(X)\equiv1\pmod{X^{p-1}}$ and $p\geq3$, \eqref{eq:global-W} therefore gives $W_1=w_p$.  Reducing \eqref{eq:global-log} modulo $X^p$ and using Proposition~\ref{prop:PhiBern} gives \eqref{eq:old-logform}.  Avitabile and Mattarei proved that \eqref{eq:old-logform} is equivalent to \eqref{eq:AMconj}.
\end{proof}

Thus there is no separate obstruction in the range $2p\leq k<3p$, or in any later range: all values are encoded by a single explicit power series.  The proof is based on the exact $p$-section
\[
 C(Y):=\sum_{n\geq0}u_{np}Y^n
\]
rather than on successive truncations of it.

The global formula admits an explicit coefficient extraction in every degree. Define
\begin{equation}\label{eq:lambdadef}
 \lambda_1=w_p,\qquad
 \lambda_d=\frac{B_{p-d}}d\quad(2\leq d\leq p-1).
\end{equation}
For a nonnegative integer
\[
 n=\sum_{i\geq0}d_i(n)p^i,\qquad 0\leq d_i(n)<p,
\]
write
\[
 s_p(n)=\sum_{i\geq0}d_i(n),\qquad
 f_p(n)=\prod_{i\geq0}d_i(n)!\in\Fp^\times.
\]
Set $\rho_p(0)=0$ and, for $n>0$,
\begin{equation}\label{eq:rho-def}
 \rho_p(n)=
 \begin{cases}
 \displaystyle \lambda_{s_p(n)}\frac{s_p(n)!}{f_p(n)},&1\leq s_p(n)\leq p-1,\\[6pt]
 \displaystyle -\frac1{f_p(n)},&s_p(n)=p,\\[6pt]
 0,&s_p(n)>p.
 \end{cases}
\end{equation}
Also set
\begin{equation}\label{eq:delta-def}
 \delta_p(0)=1,\qquad
 \delta_p(m)=-\binom{pm-1}{m-1}\in\Fp\quad(m\geq1).
\end{equation}

\begin{theorem}[Digital formula in all degrees]\label{thm:digital}
For every $k\geq0$,
\begin{equation}\label{eq:digital-W}
 W_k=\sum_{m=0}^{\lfloor k/(p-1)\rfloor}
 \delta_p(m)\,\rho_p\bigl(k-m(p-1)\bigr)\qquad\text{in }\Fp.
\end{equation}
Moreover, if
\[
 m-1=\alpha_0+\alpha_1p+\alpha_2p^2+\cdots,\qquad 0\leq \alpha_i<p,
\]
then Lucas' theorem gives
\begin{equation}\label{eq:delta-lucas}
 \delta_p(m)=-\binom{p-1}{\alpha_0}\binom{\alpha_0}{\alpha_1}
 \binom{\alpha_1}{\alpha_2}\cdots\qquad(m\geq1).
\end{equation}
In particular, $\delta_p(m)=0$ unless $\alpha_0\geq\alpha_1\geq\alpha_2\geq\cdots$.
\end{theorem}

Thus \eqref{eq:digital-W} is a finite scalar formula for every weighted convolution, expressed entirely through base-$p$ digits and the constants $\lambda_1,\ldots,\lambda_{p-1}$.

\begin{theorem}[Automaticity and finite-state evaluation]\label{thm:automatic}
The series $\Wcal(X)$ is algebraic over $\Fp(X)$, with
\[
 [\Fp(X,\Wcal):\Fp(X)]\leq p^2.
\]
Consequently the sequence $(W_k)_{k\geq0}$ is $p$-automatic.  More concretely, for each fixed $p$ there is an explicit finite-state digit procedure, given in Theorem~\ref{thm:weighted-automaton}, which evaluates $W_k$ from the base-$p$ expansion of $k$; after $p$-dependent preprocessing this requires $O(\log_p(k+2))$ state transitions.
\end{theorem}

For brevity, call the ranges $jp\leq k<(j+1)p$ the successive blocks of the weighted-convolution sequence.  In the range below $p^2-1$ the digital formula collapses to a particularly simple expression.  For $m\geq1$ put
\[
 \mathfrak h_m^{(1)}=\sum_{r=1}^m\frac1r\in\Fp,
\]
and for $j\geq0$ set
\begin{equation}\label{eq:cjdef}
 c_j=\begin{cases}
 (-1)^{j-2}\mathfrak h_{j-1}^{(1)},&2\leq j\leq p-1,\\
 0,&j=0,1\text{ or }j\geq p.
 \end{cases}
\end{equation}

\begin{theorem}[Explicit weighted convolutions below $p^2-1$]\label{thm:explicit-p2}
Let $1\leq k<p^2-1$, and write uniquely
\[
 k=d+j(p-1),\qquad 1\leq d\leq p-1,\quad j\geq0.
\]
Then
\begin{equation}\label{eq:explicit-p2}
 \boxed{W_k=\binom{d-1}{j}\lambda_d+\mathbf 1_{d=1}c_j}
 \qquad\text{in }\Fp,
\end{equation}
where $\binom{d-1}{j}=0$ for $j>d-1$.
\end{theorem}

\begin{corollary}[Second block]\label{cor:first-second}
For $p\leq k\leq2p-1$,
\[
 W_k=\begin{cases}
 0,&k=p,\\[4pt]
 \dfrac{k-p}{k-p+1}B_{2p-1-k},&p<k<2p-1,\\[7pt]
 1,&k=2p-1.
 \end{cases}
\]
\end{corollary}

\begin{corollary}[Third block]\label{cor:third-block}
For $p\geq5$ and $2p\leq k\leq3p-1$,
\[
 W_k=\begin{cases}
 \dfrac{(k-2p)(k-2p+1)}{2(k-2p+2)}B_{3p-2-k},&2p\leq k\leq3p-3,\\[8pt]
 -\dfrac32,&k=3p-2,\\[6pt]
 0,&k=3p-1.
 \end{cases}
\]
For $p=3$ the same displayed values hold in $\mathbf F_3$ by Theorem~\ref{thm:global}.
\end{corollary}

\subsection*{\texorpdfstring{Part II: $p$-adic recovery}{Part II: p-adic recovery}}
The same exact $p$-section equation also addresses the original coefficient problem, independently of the weighted-convolution consequences.  The weighted and unweighted convolution identities make the obstruction visible: together they recursively determine $a_{kp}$ except when $k$ is an odd multiple of $p$, exactly the characteristic-$p$ blind spot of the logarithmic derivative.  The $p$-adic form of the exact operator equation removes this blind spot at every scale.

\begin{theorem}[All-orders logarithmic lifting]\label{thm:all-orders}
Let
\[
 C(Y)=\sum_{n\geq0}u_{np}Y^n\in\Zp[[Y]],\qquad
 L(Y)=\frac{\theta C(Y)}{C(Y)}.
\]
There is an explicit $1$-Lipschitz map
\[
 \mathcal T_p:\Zp[[Y]]\longrightarrow\Zp[[Y]],
\]
defined solely from $q$, $P_p$ and $\theta$, such that
\begin{equation}\label{eq:intro-fixed}
 L=L_{0,p}+p\mathcal T_p(L),\qquad
 L_{0,p}=\Phi_p(-q)-D_p.
\end{equation}
Thus the fixed-point map $H\mapsto L_{0,p}+p\mathcal T_p(H)$ is a strict $p$-adic contraction.  If
\[
 L^{(0)}=L_{0,p},\qquad
 L^{(j+1)}=L_{0,p}+p\mathcal T_p(L^{(j)}),
\]
then
\[
 L^{(r-1)}\equiv L\pmod{p^r}\qquad(r\geq1).
\]
\end{theorem}

\begin{theorem}[Arbitrary-precision coefficient recovery]\label{thm:recovery}
Let $R\geq t\geq1$.  The logarithmic data $\theta C/C\pmod{p^R}$ determine, by an explicit finite divisor recursion in big Witt coordinates,
\[
 C(Y)\pmod{(p^t,Y^{p^{R-t+1}})}.
\]
Consequently every $p$-adic coefficient $u_{kp}$ is constructively recoverable to arbitrary precision: for prescribed $k$ and $t$, choose $R$ with $k<p^{R-t+1}$ and use Theorem~\ref{thm:all-orders}.  Taking $t=1$ recovers every residue $a_{kp}$.  Together with the usual coefficient recurrence at indices not divisible by $p$, this gives a purely modular algorithm for every coefficient of the Artin--Hasse exponential to any prescribed $p$-adic precision.
\end{theorem}

More specifically, Section~\ref{sec:section} constructs the exact $p$-section and its conjugated hypergeometric equation, and Section~\ref{sec:operator} proves the global operator congruence modulo $p^2$.  Section~\ref{sec:global-proof} deduces \eqref{eq:global-log} and combines it with the exact product identity of Avitabile and Mattarei.  Section~\ref{sec:defect} recalls the Bernoulli description of the falling-factorial defect.  Section~\ref{sec:digital} extracts the global generating function by base-$p$ digits.  Section~\ref{sec:explicit} specializes that formula below $p^2-1$.  Section~\ref{sec:automatic} proves algebraicity, automaticity, and the finite-state evaluator.  Section~\ref{sec:recovery} develops the independent $p$-adic branch: it isolates the characteristic-$p$ blind spot, constructs the exact fixed-point map, and performs big-Witt reconstruction at arbitrary precision.  The closed third-order normal-ordering calculation is moved to Appendix~\ref{app:third-order}, where it can be consulted without interrupting the main recovery argument.  A short remark at the end of Section~\ref{sec:recovery} records the depth-dependent precision cost for iterated Cartier $p$-sections.  Section~\ref{sec:consequences} records further consequences and remarks.

As a computational check, the characteristic-$p$ formulas in Part~I, including the explicit weighted automaton of Theorem~\ref{thm:weighted-automaton}, were implemented directly from the stated recurrences and transition rules and checked against exact Artin--Hasse coefficient computations for $p=3,5,7,11,13$ over substantial finite ranges.  The fixed-point and Witt-recovery formulas in Part~II were likewise checked for $p=3,5,7$.

\paragraph{Notation guide.}
We reserve $a_n$ for the reduced Artin--Hasse coefficients.  The shared series $q$, the falling-factorial quotient $\Phi_p$, and the Frobenius defect $D_p$ are used in both branches.  In Part~I, $G$, $\mathcal W$, and $\mathcal R$ denote respectively the signed $p$-section, its weighted-convolution generating series, and the logarithmic derivative $XG'/G$.  In Part~II, $C$ is the integral $p$-section and $L=\theta C/C$ its logarithmic derivative.  Greek letters such as $\alpha_i$ and $\kappa_i$ are reserved for base-$p$ digits.

\section{\texorpdfstring{The exact $p$-section and its operator equation}{The exact p-section and its operator equation}}\label{sec:section}
Let $\theta=Y\,\dd/\dd Y$ and define
\begin{equation}\label{eq:Fdef}
 F(Y)=\sum_{m\geq0}\frac{Y^m}{(pm)!}.
\end{equation}
Also put
\begin{equation}\label{eq:Bdef}
 \mathcal B(Y)=\exp\!\left(\sum_{i\geq0}\frac{Y^{p^i}}{p^{i+1}}\right)
      =\AH(Y)^{1/p}\in\Qp[[Y]],
\end{equation}
where the fractional power is interpreted through formal exponential and logarithm.

\begin{lemma}[Exact $p$-section]\label{lem:exact-section}
If
\[
 C(Y)=\sum_{n\geq0}u_{np}Y^n\in\Zp[[Y]],
\]
then
\begin{equation}\label{eq:section-factor}
 C(Y)=\mathcal B(Y)F(Y).
\end{equation}
\end{lemma}
\begin{proof}
Separate the $e^X$ factor in the Artin--Hasse exponential:
\[
 \AH(X)=e^X\exp\!\left(\sum_{i\geq1}\frac{X^{p^i}}{p^i}\right)
       =e^X\mathcal B(X^p).
\]
The second factor contains only powers divisible by $p$.  Therefore, in extracting the terms of total degree divisible by $p$, the $e^X$ factor contributes exactly its terms $X^{pm}/(pm)!$.  Replacing $X^p$ by $Y$ gives \eqref{eq:section-factor}.
\end{proof}

The hypergeometric series $F$ satisfies an elementary differential equation.

\begin{lemma}\label{lem:F-operator}
One has
\begin{equation}\label{eq:F-operator}
 P_p(p\theta)F=YF.
\end{equation}
\end{lemma}
\begin{proof}
For $m\geq1$,
\[
 P_p(pm)\frac{Y^m}{(pm)!}
 =\frac{Y^m}{(p(m-1))!},
\]
while the $m=0$ term is annihilated because $P_p(0)=0$.  Summing over $m$ gives the result.
\end{proof}

The point of using the full factor $\mathcal B$ is that its logarithmic derivative is the full $p$-typical series $q$:
\begin{equation}\label{eq:B-log}
 p\theta\log\mathcal B(Y)=\sum_{i\geq0}Y^{p^i}=q(Y).
\end{equation}
Define
\begin{equation}\label{eq:Udef}
 U=p\theta-q(Y),
\end{equation}
where $q(Y)$ denotes multiplication by the series $q(Y)$ when it occurs inside an operator expression.

\begin{proposition}[Exact conjugated equation]\label{prop:conjugated}
The integral series $C(Y)$ satisfies
\begin{equation}\label{eq:exact-U}
 P_p(U)C=YC
\end{equation}
as an identity in $\Qp[[Y]]$.
\end{proposition}
\begin{proof}
Let $M_h$ denote multiplication by a series $h$. Equation \eqref{eq:B-log} gives
\[
 M_{\mathcal B}(p\theta)M_{\mathcal B}^{-1}=p\theta-M_q=U.
\]
Since $C=\mathcal B F$, conjugating \eqref{eq:F-operator} by multiplication by $\mathcal B$ gives \eqref{eq:exact-U}.
\end{proof}

\section{\texorpdfstring{A global operator congruence modulo $p^2$}{A global operator congruence modulo p squared}}\label{sec:operator}
The next lemma is the mechanism which removes the need for blockwise truncation.

\begin{lemma}[Operator Frobenius congruence]\label{lem:operator-frob}
Let $f(Y)\in\Zp[[Y]]$, let $\mu$ be multiplication by $f$, and put $\delta=p\theta$.  Then, as operators on $\Zp[[Y]]$,
\begin{equation}\label{eq:operator-frob}
 (\mu+\delta)^p\equiv\mu^p\pmod{p^2}.
\end{equation}
\end{lemma}
\begin{proof}
Expand $(\mu+\delta)^p$ as a sum of noncommutative words.  Every word containing at least two occurrences of $\delta=p\theta$ is visibly divisible by $p^2$ as an endomorphism of $\Zp[[Y]]$.  The word containing no $\delta$ is $\mu^p$.

It remains to consider the sum of the words containing exactly one $\delta$:
\[
 S=\sum_{i=0}^{p-1}\mu^i\delta\mu^{p-1-i}.
\]
Since
\[
 \delta\mu^r=\mu^r\delta+pr\,\mu^{r-1}M_{\theta f},
\]
where $M_{\theta f}$ denotes multiplication by $\theta f$, we obtain
\[
 \begin{aligned}
 S
 &=p\mu^{p-1}\delta
   +p\left(\sum_{r=0}^{p-1}r\right)\mu^{p-2}M_{\theta f}\\
 &=p^2\mu^{p-1}\theta
   +p^2\frac{p-1}{2}\mu^{p-2}M_{\theta f}.
 \end{aligned}
\]
Because $p$ is odd, this is divisible by $p^2$.  Thus all terms except $\mu^p$ vanish modulo $p^2$.
\end{proof}

Lemma~\ref{lem:fixed-operators} below contains this congruence as the case $H=0$ and gives the normalized form used in the all-orders fixed-point argument.

Applying this with $f=-q(Y)$ gives the global replacement for the two truncated congruences used in the blockwise approach.

\begin{corollary}\label{cor:Up}
For the operator $U=p\theta-q(Y)$ of \eqref{eq:Udef},
\begin{equation}\label{eq:Up}
 U^p\equiv-q(Y)^p\pmod{p^2}
\end{equation}
as operators on $\Zp[[Y]]$.
\end{corollary}
\begin{proof}
Apply Lemma~\ref{lem:operator-frob} with $f=-q(Y)$. Since $p$ is odd, $\mu^p=M_{(-q)^p}=M_{-q^p}$.
\end{proof}

The second correction which previously appeared one block at a time is naturally packaged by the Frobenius defect
\begin{equation}\label{eq:Ddef}
 D_p(Y):=\frac{q(Y)^p-q(Y^p)}{p}
        =\frac{q(Y)^p-q(Y)+Y}{p}\in\mathbf Z[[Y]].
\end{equation}
The integrality follows from the freshman's-dream congruence $q(Y)^p\equiv q(Y^p)\pmod p$.

\section{The global logarithmic derivative and weighted convolution}\label{sec:global-proof}
Recall the normalized Fermat defect
\begin{equation}\label{eq:PhiDef}
 \Phi_p(Z)=\frac{P_p(Z)-(Z^p-Z)}{p}\in\mathbf Z[Z],
 \qquad \varphi_p(Z):=\overline{\Phi_p(Z)}\in\Fp[Z].
\end{equation}
Thus
\begin{equation}\label{eq:P-decompose}
 P_p(Z)=Z^p-Z+p\Phi_p(Z).
\end{equation}

\begin{proposition}[Global logarithmic derivative]\label{prop:global-log}
In $\Fp[[X]]$,
\begin{equation}\label{eq:global-intermediate}
 X\frac{G'(X)}{G(X)}
 =\varphi_p(q(X))+\overline{D_p}(X)
 =\overline{\mathcal H_p(X)}.
\end{equation}
\end{proposition}
\begin{proof}
Substitute \eqref{eq:P-decompose} into the exact equation \eqref{eq:exact-U}.  Since $UC=p\theta C-qC$, we obtain
\[
 p\theta C=U^pC+(q-Y)C+p\Phi_p(U)C.
\]
The series $C$ has coefficients in $\Zp$, so Corollary~\ref{cor:Up} may be applied to it.  Modulo $p^2$ this gives
\[
 \begin{aligned}
 p\theta C
 &\equiv-q^pC+(q-Y)C+p\Phi_p(U)C\\
 &=-pD_p(Y)C+p\Phi_p(U)C.
 \end{aligned}
\]
Divide by $p$ and reduce modulo $p$.  Since $U\equiv-q(Y)\pmod p$,
\begin{equation}\label{eq:C-log}
 \theta \overline C=\bigl(\varphi_p(-q(Y))-\overline{D_p}(Y)\bigr)\overline C
 \qquad\text{in }\Fp[[Y]].
\end{equation}

Reducing $C$ modulo $p$ and replacing $Y$ by $-X$ gives $\overline C(-X)=G(X)$.  Because $p$ is odd,
\[
 q(-X)=-q(X),\qquad D_p(-X)=-D_p(X).
\]
Hence \eqref{eq:C-log} becomes
\[
 X\frac{G'(X)}{G(X)}=\varphi_p(q(X))+\overline{D_p}(X)
 \qquad\text{in }\Fp[[X]].
\]
Finally, using $q(X^p)=q(X)-X$,
\[
 \begin{aligned}
 \Phi_p(q)+D_p
 &=\frac{P_p(q)-(q^p-q)}{p}
   +\frac{q^p-q+X}{p}\\
 &=\frac{P_p(q)+X}{p}=\mathcal H_p(X),
 \end{aligned}
\]
before reduction modulo $p$.  This proves the proposition.
\end{proof}

To pass from the logarithmic derivative to the weighted convolution we use the exact product identity recorded in \cite[Proposition~1]{AM}, where it is quoted from \cite[Proposition~2, Eq.~(4)]{AMLag}.\footnote{In \cite{AMLag}, the series denoted $T$ is our $q$ with the variable replaced by its $p$th power. Consequently their product identity becomes \eqref{eq:AM-global-product} after renaming that $p$th power as the variable.}  In our notation it is
\begin{equation}\label{eq:AM-global-product}
 G(X)G(-X)q(X)=X
 \qquad\text{in }\Fp[[X]].
\end{equation}

\begin{proof}[Proof of Theorem~\ref{thm:global}]
The logarithmic-derivative identity \eqref{eq:global-log} is Proposition~\ref{prop:global-log}.  Moreover,
\[
 [X^k]\bigl(XG'(X)G(-X)\bigr)
 =\sum_{r=0}^k(-1)^r r\,a_{rp}a_{(k-r)p}=W_k.
\]
Therefore, by \eqref{eq:AM-global-product},
\[
 \begin{aligned}
 \Wcal(X)
 &=XG'(X)G(-X)\\
 &=\left(X\frac{G'(X)}{G(X)}\right)G(X)G(-X)\\
 &=\overline{\mathcal H_p(X)}\frac{X}{q(X)},
 \end{aligned}
\]
which is \eqref{eq:global-W}.  Coefficient extraction gives \eqref{eq:all-k-coeff}.
\end{proof}

\begin{remark}\label{rem:two-defects}
The formula
\[
 \overline{\mathcal H_p(X)}=\varphi_p(q(X))+\overline{D_p}(X)
\]
separates two first-order defects.  The polynomial $\Phi_p$ measures the failure of the falling factorial $P_p(Z)$ to equal $Z^p-Z$ modulo $p^2$, while $D_p$ measures the failure of the integral series $q(X)$ to satisfy Frobenius exactly before reduction modulo $p$.  In the blockwise calculation these two corrections appeared at different-looking stages; globally they are the two summands of one identity.
\end{remark}

\section{The falling-factorial defect and Bernoulli numbers}\label{sec:defect}
The coefficients of $\Phi_p$ are classical. Glaisher's congruences for sums of products of $1,2,\ldots,p-1$ include the elementary-symmetric congruences underlying the result below; see \cite{Glaisher}. We include the short proof because it converts the global generating function into explicit Bernoulli formulas.

\begin{proposition}\label{prop:PhiBern}
In $\Fp[Z]$ one has
\begin{equation}\label{eq:PhiBern}
 \varphi_p(Z)=w_pZ-\sum_{n=1}^{p-2}\frac{B_n}{n}Z^{p-n}
           =w_pZ-\gamma_p(Z).
\end{equation}
Equivalently, if
\[
 \varphi_p(Z)=\sum_{d=1}^{p-1}\lambda_dZ^d,
\]
then the $\lambda_d$ are exactly those of \eqref{eq:lambdadef}.
\end{proposition}
\begin{proof}
Let
\[
 S_n=\sum_{a=1}^{p-1}a^n,
 \qquad
 e_n=e_n(1,2,\ldots,p-1).
\]
For $1\leq n\leq p-2$, Faulhaber's formula gives
\begin{equation}\label{eq:Faulhaber}
 S_n\equiv pB_n\pmod{p^2}.
\end{equation}
Moreover, $P_p(Z)\equiv Z^p-Z\pmod p$, so $p\mid e_j$ for $0<j<p-1$.  Newton's identities read
\[
 ne_n=\sum_{j=1}^n(-1)^{j-1}e_{n-j}S_j.
\]
Every term with $j<n$ is divisible by $p^2$, because both $e_{n-j}$ and $S_j$ are divisible by $p$.  Hence
\[
 ne_n\equiv(-1)^{n-1}pB_n\pmod{p^2}.
\]
The coefficient of $Z^{p-n}$ in $P_p(Z)$ is $(-1)^ne_n$, so
\[
 [Z^{p-n}]\varphi_p(Z)=-\frac{B_n}{n}\quad\text{in }\Fp.
\]
Finally, the coefficient of $Z$ in $P_p(Z)-(Z^p-Z)$ is $(p-1)!+1$, hence the coefficient of $Z$ in $\Phi_p$ is $w_p$.  Replacing $n$ by $p-d$ gives $\lambda_d=B_{p-d}/d$ for $2\leq d\leq p-1$.
\end{proof}

\section{Digital coefficient extraction in all degrees}\label{sec:digital}
We now extract the two factors in \eqref{eq:global-W} coefficientwise. The first factor is governed by Lucas' theorem, while the second is governed by the digit sum of the exponent.

\begin{lemma}[Small powers of $q$]\label{lem:q-small}
Let $1\le r\le p-1$. For every $n\ge0$,
\begin{equation}\label{eq:q-small-coeff}
 [X^n]q(X)^r=
 \begin{cases}
 r!/f_p(n),&s_p(n)=r,\\
 0,&s_p(n)\ne r,
 \end{cases}
 \qquad\text{in }\Fp.
\end{equation}
\end{lemma}
\begin{proof}
Expanding $q(X)^r$ amounts to choosing an ordered $r$-tuple of powers of $p$. If the power $p^i$ occurs $c_i$ times, then $\sum_i c_i=r<p$ and the corresponding exponent is $\sum_i c_ip^i$. Since every $c_i<p$, no carrying occurs, so the $c_i$ are exactly the base-$p$ digits of that exponent. Thus a contribution to $X^n$ exists precisely when $s_p(n)=r$, and in that case the number of ordered choices is $r!/\prod_i d_i(n)!=r!/f_p(n)$.
\end{proof}

\begin{lemma}[Digits of the Frobenius defect]\label{lem:D-digits}
For every $n\ge0$,
\begin{equation}\label{eq:D-digit-coeff}
 [X^n]D_p(X)=
 \begin{cases}
 -1/f_p(n),&s_p(n)=p,\\
 0,&s_p(n)\ne p,
 \end{cases}
 \qquad\text{in }\Fp.
\end{equation}
\end{lemma}
\begin{proof}
Expand
\[
 q(X)^p=\sum_{c_0+c_1+\cdots=p}\frac{p!}{\prod_i c_i!}X^{\sum_i c_ip^i}.
\]
The terms with one $c_i=p$ and all other $c_j=0$ are exactly the pure powers occurring in $q(X^p)$ and therefore cancel in $q(X)^p-q(X^p)$. In every remaining composition one has $0\le c_i\le p-1$, so again there is no carrying and the $c_i$ are precisely the base-$p$ digits of the exponent. These exponents are therefore exactly the integers of digit sum $p$, each arising from a unique composition. After division by $p$, its coefficient is
\[
 \frac{(p-1)!}{\prod_i c_i!}\equiv-\frac1{\prod_i c_i!}\pmod p
\]
by Wilson's theorem.
\end{proof}

\begin{proposition}[Digital logarithmic derivative]\label{prop:digital-log}
If
\[
 \mathcal R(X):=X\frac{G'(X)}{G(X)}=\sum_{n\ge0}\rho_nX^n,
\]
then $\rho_n=\rho_p(n)$ for every $n\ge0$.
\end{proposition}
\begin{proof}
By Proposition~\ref{prop:global-log} and Proposition~\ref{prop:PhiBern},
\[
 \mathcal R(X)=\sum_{r=1}^{p-1}\lambda_rq(X)^r+\overline{D_p}(X).
\]
For a fixed $n$, Lemma~\ref{lem:q-small} shows that at most one summand in the first term contributes, namely $r=s_p(n)$ when $1\le s_p(n)\le p-1$. Lemma~\ref{lem:D-digits} supplies exactly the case $s_p(n)=p$. No term contributes when $s_p(n)>p$.
\end{proof}

It remains to extract $X/q(X)$. Put
\begin{equation}\label{eq:e-def}
 e_i=\frac{p^i-1}{p-1}=1+p+\cdots+p^{i-1}\quad(i\ge1),\qquad e_0=0,
\end{equation}
and
\begin{equation}\label{eq:S-def}
 S(Z)=\sum_{i\ge0}Z^{e_i}\in\Fp[[Z]].
\end{equation}
Then
\begin{equation}\label{eq:q-S}
 \frac{q(X)}X=S(X^{p-1}).
\end{equation}

\begin{lemma}[The reciprocal factor]\label{lem:reciprocal}
In $\Fp[[Z]]$ one has
\begin{equation}\label{eq:S-recip}
 \frac1{S(Z)}=\sum_{m\ge0}\delta_p(m)Z^m,
\end{equation}
where $\delta_p$ is defined by \eqref{eq:delta-def}. Equivalently, formula \eqref{eq:delta-lucas} holds.
\end{lemma}
\begin{proof}
The exponents satisfy $e_{i+1}=1+pe_i$. Hence, in characteristic $p$,
\begin{equation}\label{eq:S-functional}
 S(Z)=1+ZS(Z)^p.
\end{equation}
Write $\Delta(Z)=S(Z)^{-1}$. Then
\begin{equation}\label{eq:Delta-AS}
 \Delta(Z)^p-\Delta(Z)^{p-1}+Z=0.
\end{equation}
To determine its coefficients without dividing by integers in $\Fp$, work first over $\mathbf Z$. Let $\mathcal U(Z)\in\mathbf Z[[Z]]$ be the unique solution with $\mathcal U(0)=0$ of
\begin{equation}\label{eq:U-functional}
 \mathcal U=Z(1-\mathcal U)^{1-p}.
\end{equation}
Then $1-\mathcal U$ satisfies \eqref{eq:Delta-AS} over $\mathbf Z$, and its reduction modulo $p$ has constant term $1$. Equation~\eqref{eq:Delta-AS} has a unique solution in $1+Z\Fp[[Z]]$: its derivative with respect to $\Delta$ at $(\Delta,Z)=(1,0)$ is $1$. Hence
\[
 \Delta(Z)=1-\mathcal U(Z).
\]
Lagrange inversion gives, for $m\ge1$,
\begin{equation}\label{eq:lagrange-Am}
 [Z^m]\mathcal U(Z)=\frac1m\binom{pm-2}{m-1}=:A_m\in\mathbf Z.
\end{equation}
The identity
\[
 (pm-1)A_m=(p-1)\binom{pm-1}{m-1}
\]
shows, after reduction modulo $p$, that
\[
 A_m\equiv\binom{pm-1}{m-1}\pmod p.
\]
Consequently
\[
 [Z^m]\Delta(Z)=-\binom{pm-1}{m-1}=\delta_p(m),
\]
which proves \eqref{eq:S-recip}.

Finally, if $m-1=\sum_i \alpha_ip^i$, then
\[
 pm-1=(p-1)+\alpha_0p+\alpha_1p^2+\cdots.
\]
Lucas' theorem therefore gives
\[
 \binom{pm-1}{m-1}
 \equiv\binom{p-1}{\alpha_0}\binom{\alpha_0}{\alpha_1}\binom{\alpha_1}{\alpha_2}\cdots\pmod p,
\]
which is \eqref{eq:delta-lucas}. The product vanishes exactly when some $\alpha_{i+1}>\alpha_i$.
\end{proof}

\begin{proof}[Proof of Theorem~\ref{thm:digital}]
Proposition~\ref{prop:digital-log} gives
\[
 \overline{\mathcal H_p(X)}=\mathcal R(X)=\sum_{n\ge0}\rho_p(n)X^n.
\]
By \eqref{eq:q-S} and Lemma~\ref{lem:reciprocal},
\[
 \frac{X}{q(X)}=\frac1{S(X^{p-1})}=\sum_{m\ge0}\delta_p(m)X^{m(p-1)}.
\]
Multiplication in Theorem~\ref{thm:global} and coefficient extraction give \eqref{eq:digital-W}.
\end{proof}

\section{\texorpdfstring{Explicit weighted convolutions below $p^2-1$}{Explicit weighted convolutions below p squared minus 1}}\label{sec:explicit}
We now extract Theorem~\ref{thm:explicit-p2} from the global formula. The harmonic-number correction comes entirely from the Frobenius defect $D_p$.

\begin{proof}[Proof of Theorem~\ref{thm:explicit-p2}]
Put $z=X^{p-1}$. For coefficients of degree below $p^2-1$ we may use
\begin{equation}\label{eq:q-low}
 q(X)=X(1+z)+O(X^{p^2}),\qquad \frac{X}{q(X)}=\frac1{1+z}+O(X^{p^2-1}).
\end{equation}
By Proposition~\ref{prop:global-log} and Theorem~\ref{thm:global},
\begin{equation}\label{eq:split-W}
 \Wcal(X)=\frac{X}{q(X)}\varphi_p(q(X))+\frac{X}{q(X)}\overline{D_p}(X).
\end{equation}
Write $\varphi_p(Z)=\sum_{d=1}^{p-1}\lambda_dZ^d$. The first term of \eqref{eq:split-W}, using \eqref{eq:q-low}, is
\begin{equation}\label{eq:Phi-contrib}
 \frac{X}{q}\varphi_p(q)\equiv\sum_{d=1}^{p-1}\lambda_dX^d(1+z)^{d-1}\pmod{X^{p^2-1}}.
\end{equation}
Thus its coefficient at $X^{d+j(p-1)}$ is
\[
 \binom{d-1}{j}\lambda_d.
\]

For the second term, recall $D_p=(q^p-q(X^p))/p$. Up to degrees which can affect \eqref{eq:split-W} below $p^2-1$, the pure terms $X^p$ and $X^{p^2}$ cancel, leaving
\[
 D_p(X)\equiv X^p\sum_{r=1}^{p-1}\frac1p\binom pr z^r\pmod{X^{p^2}}.
\]
Since
\[
 \frac1p\binom pr=\frac1r\binom{p-1}{r-1}\equiv\frac{(-1)^{r-1}}r\pmod p,
\]
we obtain
\begin{equation}\label{eq:D-contrib}
 \frac{X}{q}D_p(X)\equiv X\frac{z}{1+z}\sum_{r=1}^{p-1}\frac{(-1)^{r-1}}r z^r\pmod{(p,X^{p^2-1})}.
\end{equation}
For $2\le j\le p-1$, the coefficient of $z^j$ in the factor following $X$ is
\[
 \sum_{r=1}^{j-1}\frac{(-1)^{r-1}}r(-1)^{j-1-r}
 =(-1)^{j-2}\sum_{r=1}^{j-1}\frac1r
 =(-1)^{j-2}\mathfrak h_{j-1}^{(1)}.
\]
The coefficient at $j=p$ is $(-1)^{p-2}\mathfrak h_{p-1}^{(1)}=0$ in $\Fp$. There are no contributions for $j=0,1$. Thus the Frobenius-defect term is supported, in the stated range, exactly on the ray $d=1$ and contributes $c_j$. Combining this with \eqref{eq:Phi-contrib} proves \eqref{eq:explicit-p2}.
\end{proof}

\begin{proof}[Proof of Corollary~\ref{cor:first-second}]
For $k=p$, one has $(d,j)=(1,1)$ in Theorem~\ref{thm:explicit-p2}, giving $W_p=0$. If $p<k<2p-1$, put $d=k-p+1$ and $j=1$. Then
\[
 W_k=(d-1)\frac{B_{p-d}}d
 =\frac{k-p}{k-p+1}B_{2p-1-k}.
\]
Finally, $2p-1=1+2(p-1)$, so $d=1$, $j=2$, and $W_{2p-1}=c_2=H_1^{(1)}=1$.
\end{proof}

\begin{proof}[Proof of Corollary~\ref{cor:third-block}]
Assume first that $p\ge5$, so $3p-1<p^2-1$. For $2p\le k\le3p-3$, write
\[
 d=k-2p+2,\qquad j=2.
\]
Then $2\le d\le p-1$, and Theorem~\ref{thm:explicit-p2} gives
\[
 W_k=\binom{d-1}{2}\frac{B_{p-d}}d
 =\frac{(k-2p)(k-2p+1)}{2(k-2p+2)}B_{3p-2-k}.
\]
At $k=3p-2$ we have $(d,j)=(1,3)$, so
\[
 W_{3p-2}=c_3=-H_2^{(1)}=-\frac32.
\]
At $k=3p-1$ we have $(d,j)=(2,3)$, and the binomial coefficient vanishes. For $p=3$, direct coefficient extraction from \eqref{eq:global-W} gives $W_6=W_7=W_8=0$, which is the reduction of the same displayed formulas.
\end{proof}

\section{Automaticity and finite-state evaluation}\label{sec:automatic}
We now show that the digital formula has a finite-state interpretation. Recall that a sequence with values in a finite set is called $p$-automatic if a deterministic finite automaton with output can recover its $n$th term from the base-$p$ expansion of $n$. Christol's theorem \cite{Christol} states that a sequence $(c_n)_{n\ge0}$ with values in $\Fp$ is $p$-automatic if and only if its generating series $\sum_{n\ge0}c_nX^n$ is algebraic over $\Fp(X)$.

Recall
\begin{equation}\label{eq:phi-def}
 \varphi_p(Z)=\overline{\Phi_p(Z)}=\sum_{r=1}^{p-1}\lambda_rZ^r\in\Fp[Z].
\end{equation}
By Lemma~\ref{lem:D-digits}, the reduction of the Frobenius defect modulo $p$ is
\begin{equation}\label{eq:Dbar-digits}
 \overline{D_p}(X)=-\sum_{s_p(n)=p}\frac{X^n}{f_p(n)}.
\end{equation}
The series $\overline{D_p}$ satisfies a simple Artin--Schreier equation over the field generated by $q$.

\begin{proposition}[Artin--Schreier equation for the Frobenius defect]\label{prop:Dbar-AS}
Let
\[
 \Omega_p(U,V)=\frac{U^p+V^p-(U+V)^p}{p}\in\mathbf Z[U,V]
\]
be the standard first $p$-typical Witt addition correction polynomial, and define
\begin{equation}\label{eq:Psi-def}
 \Psi_p(X,Z)=\overline{\Omega_p(X,Z-X)}
 =\sum_{d=1}^{p-1}\frac{X^d(Z-X)^{p-d}}{d!(p-d)!}\in\Fp[X,Z].
\end{equation}
Equivalently,
\[
 \Psi_p(X,Z)=-\overline{\frac{Z^p-X^p-(Z-X)^p}{p}}.
\]
Then
\begin{equation}\label{eq:Dbar-AS}
 \overline{D_p}(X)^p-\overline{D_p}(X)=\Psi_p(X,q(X)).
\end{equation}
\end{proposition}
\begin{proof}
Partition the sum \eqref{eq:Dbar-digits} according to the least significant base-$p$ digit $d=d_0(n)$. If $d=0$, write $n=pn'$. The condition $s_p(n)=p$ becomes $s_p(n')=p$, and this part of \eqref{eq:Dbar-digits} is $\overline{D_p}(X^p)$.

If $1\le d\le p-1$, write $n=d+pn'$. Then $s_p(n')=p-d<p$, and Lemma~\ref{lem:q-small} gives
\[
 \sum_{s_p(n')=p-d}\frac{X^{pn'}}{f_p(n')}
 =\frac{q(X^p)^{p-d}}{(p-d)!}.
\]
Consequently
\begin{equation}\label{eq:Dbar-AS-step}
 \overline{D_p}(X)=\overline{D_p}(X^p)-\sum_{d=1}^{p-1}
 \frac{X^dq(X^p)^{p-d}}{d!(p-d)!}.
\end{equation}
In characteristic $p$, $\overline{D_p}(X^p)=\overline{D_p}(X)^p$ and $q(X^p)=q(X)-X$. Rearranging \eqref{eq:Dbar-AS-step} gives \eqref{eq:Dbar-AS}.
\end{proof}

\begin{remark}[Witt interpretation]
The closed form
\[
 \Psi_p(X,Z)=\overline{\Omega_p(X,Z-X)}
 =-\overline{\frac{Z^p-X^p-(Z-X)^p}{p}}
\]
identifies the Frobenius correction with the first $p$-typical Witt addition correction for the decomposition $Z=X+(Z-X)$. Thus the Artin--Schreier equation above is naturally Witt-theoretic rather than an accidental coefficient calculation; this also foreshadows the ghost-coordinate reconstruction in Section~\ref{sec:recovery}.
\end{remark}

\begin{corollary}[Algebraicity]\label{cor:R-AS}
Put
\begin{equation}\label{eq:Theta-def}
 \Theta_p(X,Z)=\Psi_p(X,Z)+\varphi_p(Z-X)-\varphi_p(Z).
\end{equation}
Then
\begin{equation}\label{eq:R-AS}
 \mathcal R(X)^p-\mathcal R(X)=\Theta_p(X,q(X)),\qquad \mathcal R(X)=X\frac{G'(X)}{G(X)},
\end{equation}
and
\begin{equation}\label{eq:W-algebraic}
 (q(X)-X)\Wcal(X)^p-q(X)X^{p-1}\Wcal(X)=X^p\Theta_p(X,q(X)).
\end{equation}
In particular, $\Wcal$ is algebraic over $\Fp(X)$ of degree at most $p^2$.
\end{corollary}
\begin{proof}
Proposition~\ref{prop:global-log} and \eqref{eq:phi-def} give
\[
 \mathcal R(X)=\varphi_p(q(X))+\overline{D_p}(X).
\]
Using Proposition~\ref{prop:Dbar-AS} and $q^p=q-X$ yields
\[
 \mathcal R^p-\mathcal R=\Psi_p(X,q)+\varphi_p(q)^p-\varphi_p(q)
      =\Psi_p(X,q)+\varphi_p(q-X)-\varphi_p(q),
\]
which is \eqref{eq:R-AS}. Since $\Wcal=X\mathcal R/q$, substitution in \eqref{eq:R-AS} gives \eqref{eq:W-algebraic}. Finally $q$ satisfies $q^p-q+X=0$, so $[\Fp(X,q):\Fp(X)]\le p$. Equation~\eqref{eq:R-AS} has degree $p$ in $\mathcal R$, hence adjoining $\mathcal R$ increases the degree by at most another factor $p$. As $\Wcal=X\mathcal R/q$ belongs to $\Fp(X,q,\mathcal R)$, the asserted degree bound follows.
\end{proof}

\begin{proof}[Proof of the automaticity assertion in Theorem~\ref{thm:automatic}]
Corollary~\ref{cor:R-AS} and Christol's theorem \cite{Christol} imply that $(W_k)_{k\ge0}$ is $p$-automatic.
\end{proof}

We next give a direct finite-state evaluator. This is more explicit than the existence statement furnished by Christol's theorem. Define
\begin{equation}\label{eq:h-def}
 \omega_p(s)=
 \begin{cases}
 0,&s=0,\\
 \lambda_s s!,&1\le s\le p-1,\\
 -1,&s=p.
 \end{cases}
\end{equation}
Thus, directly from the definition of $\rho_p$,
\begin{equation}\label{eq:rho-h}
 \rho_p(n)=\frac{\omega_p(s_p(n))}{f_p(n)}
 \quad\text{whenever }s_p(n)\le p,
\end{equation}
and $\rho_p(n)=0$ when $s_p(n)>p$.

\begin{theorem}[Explicit weighted automaton]\label{thm:weighted-automaton}
Let $k\ge p-1$ and put
\[
 K=k-(p-1)=\sum_{i=0}^{L-1}\kappa_ip^i,
\]
where $0\le \kappa_i\le p-1$; if $K=0$, take $L=1$ and $\kappa_0=0$. Use the state set
\[
 \mathcal S_p=\{(c,\alpha,s):0\le c,\alpha\le p-1,\ 0\le s\le p\}.
\]
The evaluator is the following finite algorithm.
\begin{enumerate}
\item \emph{Initialize.} Let $v_0$ be the vector supported at $(0,p-1,0)$ with value $1$.
\item \emph{Process digit $\kappa_i$.} For $i=0,\ldots,L-1$, first set every entry of $v_{i+1}$ equal to zero. For each state $(c,\alpha,s)$ and each $\alpha'$ with $0\le \alpha'\le \alpha$, let $d\in\{0,\ldots,p-1\}$ be the least residue of
\begin{equation}\label{eq:automaton-d}
 d\equiv \kappa_i-c+\alpha'\pmod p,
\end{equation}
and put
\begin{equation}\label{eq:automaton-carry}
 c'=\frac{d+(p-1)\alpha'+c-\kappa_i}{p}.
\end{equation}
If $s+d\le p$, add
\begin{equation}\label{eq:automaton-transition}
 v_i(c,\alpha,s)\binom{\alpha}{\alpha'}\frac1{d!}
\end{equation}
to $v_{i+1}(c',\alpha',s+d)$. All arithmetic is in $\Fp$.
\item \emph{Output.} After all $L$ digits have been read, return
\begin{equation}\label{eq:automaton-final}
 W_k=\rho_p(k)-\sum_{\alpha=0}^{p-1}\sum_{s=0}^p \omega_p(s)v_L(0,\alpha,s).
\end{equation}
\end{enumerate}
For $0\le k<p-1$ one simply has $W_k=\rho_p(k)$.
\end{theorem}
\begin{proof}
Separate the $m=0$ term in \eqref{eq:digital-W}. For $m\ge1$ write $t=m-1$ and put
\[
 n=k-m(p-1).
\]
Then
\begin{equation}\label{eq:carry-identity}
 n+(p-1)t=K.
\end{equation}
Write
\[
 t=\sum_{i=0}^{L-1}\alpha_ip^i,\qquad
 n=\sum_{i=0}^{L-1}d_ip^i.
\]
No higher digit of $t$ can occur, because $(p-1)t\le K<p^L$. The ordinary carries in \eqref{eq:carry-identity} satisfy
\begin{equation}\label{eq:carry-columns}
 d_i+(p-1)\alpha_i+c_i=\kappa_i+pc_{i+1},\qquad c_0=0.
\end{equation}
Given $c_i$, $\kappa_i$, and $\alpha_i$, equations~\eqref{eq:automaton-d} and \eqref{eq:automaton-carry} are exactly the unique choice of $d_i\in\{0,\ldots,p-1\}$ and the resulting carry. Moreover $0\le c_{i+1}\le p-1$.

By \eqref{eq:delta-lucas}, the contribution is zero unless
\[
 p-1\ge \alpha_0\ge \alpha_1\ge\cdots\ge0.
\]
Starting with the auxiliary previous digit $\alpha_{-1}=p-1$, the product of transition factors
\[
 \binom{\alpha_{-1}}{\alpha_0}\binom{\alpha_0}{\alpha_1}\binom{\alpha_1}{\alpha_2}\cdots
\]
is therefore $-\delta_p(t+1)$. The factors $1/d_i!$ accumulate to $1/f_p(n)$, while the state variable $s$ records $s_p(n)$. Paths with $s_p(n)>p$ may be discarded because $\rho_p(n)=0$. By \eqref{eq:rho-h}, multiplying an accepted path---that is, one with final carry $c_L=0$---by $\omega_p(s)$ gives
\[
 -\delta_p(t+1)\rho_p(n).
\]
Summing over all paths therefore yields the negative of the $m\ge1$ part of \eqref{eq:digital-W}. Adding the $m=0$ term $\rho_p(k)$ proves \eqref{eq:automaton-final}.
\end{proof}

\begin{proof}[Proof of the complexity assertion in Theorem~\ref{thm:automatic}]
For each input digit $\kappa\in\{0,\ldots,p-1\}$, the transition \eqref{eq:automaton-transition} is a fixed linear transformation of the vector space $\Fp^{\mathcal S_p}$. Thus the weighted automaton has $p^2(p+1)$ basic states and processes $O(\log_p(k+2))$ digits. For fixed $p$ this is $O(\log k)$ field operations, with a constant depending only on $p$.

If desired, one may determinize completely: the current vector $(v_i(c,\alpha,s))$ itself is a state, and there are at most
\[
 p^{p^2(p+1)}
\]
such vectors. This bound is deliberately crude and is not intended as a practical determinization procedure; the weighted form of Theorem~\ref{thm:weighted-automaton} is the practical evaluator. If desired, precomputing the actually reachable vectors and their $p$ outgoing transitions turns evaluation into one table lookup per base-$p$ digit. The separate computation of $\rho_p(k)$ is itself finite-state from its digit formula, and subtraction of the fixed integer $p-1$ is a finite-state digit operation. Hence the displayed weighted construction yields an ordinary finite automaton with output after $p$-dependent preprocessing. The carry construction naturally reads the least significant digit first; reversing the corresponding regular output languages gives an equivalent automaton in the usual most-significant-digit convention.
\end{proof}

\begin{remark}[Automatic convolution versus the Artin--Hasse coefficients]
The automaticity proved here is a property of the weighted $p$-section convolution $(W_k)_{k\geq0}$, not of the Artin--Hasse coefficient sequence itself. Kramer-Miller \cite{KramerMiller} proved that the reduction
\[
 \overline{\operatorname{AH}_p}(X)=\sum_{n\geq0}a_nX^n\in\Fp[[X]]
\]
of the Artin--Hasse exponential is transcendental over $\Fp(X)$, answering a question of Thakur. By Christol's theorem, the $p$-kernel of $(a_n)_{n\geq0}$ is therefore infinite; equivalently, $(a_n)_{n\geq0}$ is not $p$-automatic. There is consequently no conflict with Theorem~\ref{thm:automatic}: the algebraic series $\Wcal(X)$ is obtained from a particular weighted convolution of the nonautomatic Artin--Hasse coefficients. Thus this convolution exhibits substantially greater finite-state rigidity than the underlying coefficient sequence.
\end{remark}

\section{\texorpdfstring{Recovering the underlying $p$-section coefficients}{Recovering the underlying p-section coefficients}}\label{sec:recovery}
The recovery branch has three steps.  Proposition~\ref{prop:blind-spot} first uses the weighted and unweighted convolution identities to isolate exactly what is invisible in characteristic $p$; this step is motivational rather than structural.  The subsection \emph{The exact $p$-adic fixed-point map} then returns to Proposition~\ref{prop:conjugated} itself and turns the exact conjugated $p$-section equation into a strict $p$-adic fixed point for $L=\theta C/C$, independently of Theorems~\ref{thm:global}--\ref{thm:automatic}.  The following subsection, \emph{Witt-coordinate reconstruction and coefficient recovery}, converts the resulting logarithmic precision into coefficient precision by big-Witt reconstruction.  The closed expansion through $p^3$ is computationally useful but not needed for this argument, so its normal-ordering formulas are collected separately in Appendix~\ref{app:third-order}.  A final remark records the finite-depth Cartier precision descent.

Put
\[
 \beta_k=a_{kp}\in\Fp,\qquad \eta_k=[X^k]\frac{X}{q(X)}.
\]
By Lemma~\ref{lem:reciprocal},
\begin{equation}\label{eq:eta-coeff}
 \eta_k=\begin{cases}
 \delta_p(k/(p-1)),&(p-1)\mid k,\\
 0,&(p-1)\nmid k.
 \end{cases}
\end{equation}
The exact product identity and the weighted identity give two triangular recurrences.

\begin{proposition}[The exact blind spot]\label{prop:blind-spot}
Assume that $\beta_0,\ldots,\beta_{k-1}$ are known. If $p\nmid k$, then
\begin{equation}\label{eq:blind-weighted}
 \beta_k=(-1)^kk^{-1}\left(W_k-\sum_{r=1}^{k-1}(-1)^r r\,\beta_r\beta_{k-r}\right).
\end{equation}
If $p\mid k$ and $k$ is even, then
\begin{equation}\label{eq:blind-unweighted}
 \beta_k=\frac12\left(\eta_k-\sum_{r=1}^{k-1}(-1)^r\beta_r\beta_{k-r}\right).
\end{equation}
Thus the two convolution identities fail simultaneously to determine $\beta_k$ only when $k=p\ell$ with $\ell$ odd.
\end{proposition}
\begin{proof}
Separating the endpoint $r=k$ in the weighted convolution gives
\begin{equation}\label{eq:weighted-endpoint}
 W_k=(-1)^kk\beta_k+\sum_{r=1}^{k-1}(-1)^r r\,\beta_r\beta_{k-r},
\end{equation}
which gives \eqref{eq:blind-weighted} when $p\nmid k$. The product identity \eqref{eq:AM-global-product} gives
\begin{equation}\label{eq:unweighted-endpoint}
 \eta_k=(1+(-1)^k)\beta_k+\sum_{r=1}^{k-1}(-1)^r\beta_r\beta_{k-r}.
\end{equation}
For even $k$, the coefficient $2$ is invertible because $p$ is odd; for odd $k$ the endpoint disappears. The simultaneous failure is therefore exactly at odd multiples of $p$.
\end{proof}

\subsection*{\texorpdfstring{The exact $p$-adic fixed-point map}{The exact p-adic fixed-point map}}
The finite-order lifts are shadows of an exact $p$-adic fixed-point equation. Put
\begin{equation}\label{eq:f-L0}
 f(Y)=-q(Y),\qquad L_{0,p}(Y)=\Phi_p(f(Y))-D_p(Y).
\end{equation}
For $H\in\Zp[[Y]]$, let $M_H$ denote multiplication by $H$, and define
\begin{equation}\label{eq:E-T}
 E_H=\theta+M_H,\qquad T_H=M_f+pE_H.
\end{equation}
Thus $T_H$ is a differential operator on $\Zp[[Y]]$.

The two normalized perturbations which occur below are integral to all orders, not merely modulo a fixed power of $p$.
\begin{lemma}[Integral normalized perturbations]\label{lem:fixed-operators}
For every $H\in\Zp[[Y]]$ the operators
\begin{equation}\label{eq:Q-V}
 Q_f(H):=\frac{T_H^p-M_f^p}{p^2},
\end{equation}
\begin{equation}\label{eq:V-def}
 V_f(H):=\frac{\Phi_p(T_H)-M_{\Phi_p(f)}}p
\end{equation}
map $\Zp[[Y]]$ to itself. Moreover they are $1$-Lipschitz in $H$: if
\[
 H\equiv H'\pmod{p^r},
\]
then, as operators on $\Zp[[Y]]$,
\[
 Q_f(H)\equiv Q_f(H')\pmod{p^r},\qquad
 V_f(H)\equiv V_f(H')\pmod{p^r}.
\]
\end{lemma}
\begin{proof}
Write $\mu=M_f$ and $E=E_H$. In the noncommutative expansion of $(\mu+pE)^p$, let $S_j(E)$ denote the sum of the words containing exactly $j$ occurrences of $E$. Then
\[
 (\mu+pE)^p-\mu^p=pS_1(E)+\sum_{j=2}^p p^jS_j(E).
\]
Because the multiplication part of $E$ commutes with $\mu$, one has
\[
 E\mu^m=\mu^mE+m\mu^{m-1}M_{\theta f}.
\]
Hence
\[
 S_1(E)=p\mu^{p-1}E+\frac{p(p-1)}2\mu^{p-2}M_{\theta f}.
\]
It follows that the quotient in \eqref{eq:Q-V} is the integral operator
\begin{equation}\label{eq:Q-integral-form}
 Q_f(H)=\mu^{p-1}E+\frac{p-1}{2}\mu^{p-2}M_{\theta f}
       +\sum_{j=2}^p p^{j-2}S_j(E).
\end{equation}
Thus it is an integral noncommutative polynomial in $E_H$ with coefficients independent of $H$. Similarly, if $\Phi_p(Z)=\sum_n \phi_nZ^n$, then after expanding each $(\mu+pE)^n$ the difference $\Phi_p(\mu+pE)-\Phi_p(\mu)$ has a factor $p$; after division by $p$ the result is again an integral noncommutative polynomial in $E_H$.

Finally $E_H-E_{H'}=M_{H-H'}$. The difference of the values of either integral polynomial at $E_H$ and $E_{H'}$ is a sum of terms each containing at least one factor $E_H-E_{H'}$. If $H-H'\in p^r\Zp[[Y]]$, every such term maps $\Zp[[Y]]$ into $p^r\Zp[[Y]]$. This proves both Lipschitz assertions.
\end{proof}

Define
\begin{equation}\label{eq:F-map}
 \mathcal T_p(H)=\bigl(Q_f(H)+V_f(H)\bigr)1\in\Zp[[Y]].
\end{equation}
Lemma~\ref{lem:fixed-operators} shows at once that $\mathcal T_p$ is $1$-Lipschitz.

\begin{theorem}[Exact $p$-adic fixed point]\label{thm:fixed-point}
Let $L=\theta C/C$. Then
\begin{equation}\label{eq:fixed-point}
 L=L_{0,p}+p\mathcal T_p(L).
\end{equation}
It is the unique fixed point in $\Zp[[Y]]$ of $H\mapsto L_{0,p}+p\mathcal T_p(H)$. If
\begin{equation}\label{eq:fixed-iterates}
 L^{(0)}=L_{0,p},\qquad L^{(j+1)}=L_{0,p}+p\mathcal T_p(L^{(j)}),
\end{equation}
then
\begin{equation}\label{eq:fixed-precision}
 L^{(r-1)}\equiv L\pmod{p^r}\qquad(r\ge1).
\end{equation}
All calculations needed for \eqref{eq:fixed-precision} modulo $Y^N$ may be performed in the finite ring
\[
 (\mathbf Z/p^r\mathbf Z)[Y]/(Y^N),
\]
with two extra powers of $p$ retained temporarily when evaluating \eqref{eq:Q-V}.
\end{theorem}
\begin{proof}
Multiplication by the unit $C$ conjugates $\theta$ according to
\[
 C^{-1}\theta C=\theta+M_L.
\]
Since $U=p\theta+f$, Proposition~\ref{prop:conjugated} therefore becomes
\[
 P_p(T_L)1=Y.
\]
Using $P_p(Z)=Z^p-Z+p\Phi_p(Z)$ and $T_L1=f+pL$ gives
\[
 pL=T_L^p1-f-Y+p\Phi_p(T_L)1.
\]
Now
\[
 f^p-f-Y=-pD_p(Y),
\]
so after adding and subtracting $f^p$ and $p\Phi_p(f)$, and then using \eqref{eq:Q-V}--\eqref{eq:F-map}, we obtain exactly
\[
 L=\Phi_p(f)-D_p+p\mathcal T_p(L)=L_{0,p}+p\mathcal T_p(L).
\]
By Lemma~\ref{lem:fixed-operators}, $\mathcal T_p$ is $1$-Lipschitz. Thus the map $H\mapsto L_{0,p}+p\mathcal T_p(H)$ improves every $p$-adic congruence by one power of $p$. The fixed-point identity first gives $L\equiv L^{(0)}\pmod p$. If
\[
 L\equiv L^{(j)}\pmod{p^{j+1}},
\]
then
\[
 L-L^{(j+1)}=p\bigl(\mathcal T_p(L)-\mathcal T_p(L^{(j)})\bigr)
 \equiv0\pmod{p^{j+2}}.
\]
This proves \eqref{eq:fixed-precision}. If $H,H'$ are two fixed points, the same argument shows successively that $H\equiv H'\pmod{p^r}$ for every $r$; hence $H=H'$ coefficientwise.

For the final finite-truncation assertion, multiplication and $\theta$ preserve the ideal $(Y^N)$. Hence all operations may be performed in $(\mathbf Z/p^r\mathbf Z)[Y]/(Y^N)$; in the quotient defining $Q_f$ one retains two additional powers of $p$ temporarily before the exact division by $p^2$.
\end{proof}

\begin{proof}[Proof of Theorem~\ref{thm:all-orders}]
This is exactly Lemma~\ref{lem:fixed-operators} and Theorem~\ref{thm:fixed-point}.
\end{proof}

Theorem~\ref{thm:fixed-point} removes the finite-order operator obstruction altogether.  The closed third-order specialization is not needed for the general recovery theorem and is therefore deferred to Appendix~\ref{app:third-order}.  We proceed directly to the all-orders reconstruction.

\subsection*{Witt-coordinate reconstruction and coefficient recovery}
The fixed-point iteration computes logarithmic ghost components to arbitrary $p$-adic precision. We now invert those ghost components by big Witt coordinates, with a quantitative precision trade-off.

\begin{lemma}[$p$-adic logarithmic rigidity]\label{lem:witt-rigidity}
Let $R\ge1$ and $\Xi_1,\Xi_2\in1+Y\Zp[[Y]]$. If
\[
 \frac{\theta \Xi_1}{\Xi_1}\equiv\frac{\theta \Xi_2}{\Xi_2}\pmod{p^R},
\]
and
\[
 \frac{\Xi_1}{\Xi_2}=\prod_{n\ge1}(1-x_nY^n)^{-1}
\]
is its big Witt factorization, then
\begin{equation}\label{eq:witt-valuations}
 \vp(x_n)\ge\max\{R-\vp(n),0\}\qquad(n\ge1).
\end{equation}
Consequently, for $1\le t\le R$,
\begin{equation}\label{eq:witt-rigidity}
 \Xi_1\equiv \Xi_2\pmod{(p^t,Y^{p^{R-t+1}})}.
\end{equation}
In particular, a logarithmic derivative known modulo $p^R$ determines the reduction modulo $p$ of the underlying unit series through degree $p^R-1$.
\end{lemma}
\begin{proof}
Put $\Xi=\Xi_1/\Xi_2$. Taking logarithmic derivatives of the big Witt factorization gives
\begin{equation}\label{eq:ghost}
 [Y^N]\frac{\theta \Xi}{\Xi}=\sum_{d\mid N}d\,x_d^{N/d}.
\end{equation}
We prove \eqref{eq:witt-valuations} by induction on $n$. Put $s=\vp(n)$. If $s\ge R$ there is nothing to prove. For every proper divisor $d$ of $n$, write $h=n/d\ge2$ and $e=\vp(d)$. By induction, if $e<R$ then
\[
 \vp(dx_d^h)\ge e+h(R-e)\ge R,
\]
while if $e\ge R$ the factor $d$ already contributes $p^R$. Thus every proper-divisor term in \eqref{eq:ghost} is divisible by $p^R$. The left side is also divisible by $p^R$, so $nx_n$ is divisible by $p^R$ and \eqref{eq:witt-valuations} follows.

Now let $n<p^{R-t+1}$. Then $\vp(n)\le R-t$, so \eqref{eq:witt-valuations} gives $x_n\in p^t\Zp$. In the product for $\Xi$, all factors with index at least $p^{R-t+1}$ contribute only terms of that degree or higher, while every smaller-index factor is congruent to $1$ modulo $p^t$. This proves \eqref{eq:witt-rigidity}.
\end{proof}

Lemma~\ref{lem:witt-rigidity} converts the arbitrary logarithmic precision furnished by Theorem~\ref{thm:fixed-point} into coefficient precision for $C$. The next theorem makes that reconstruction completely explicit.

\begin{theorem}[Witt-coordinate recovery at arbitrary precision]\label{thm:witt-recovery}
Fix $R\ge1$, and write
\[
 L^{(R-1)}(Y)=\sum_{n\ge1}\ell_nY^n,\qquad \ell_n\in\mathbf Z/p^R\mathbf Z.
\]
For $1\le n<p^R$, write uniquely $n=p^sm$ with $0\le s<R$ and $p\nmid m$. Define $x_n$ recursively in increasing order of $n$, with $x_n$ taken modulo $p^{R-s}$, by
\begin{align}
 x_{p^sm}=m^{-1}\Bigg[&\frac1{p^s}\left(\ell_{p^sm}
 -\sum_{j=0}^{s-1}p^j\sum_{d\mid m}d\,x_{p^jd}^{p^{s-j}m/d}\right)\notag\\
 &-\sum_{\substack{d\mid m\\d<m}}d\,x_{p^sd}^{m/d}\Bigg]
 \pmod{p^{R-s}}.\label{eq:witt-recursion}
\end{align}
For $s=0$ the first double sum is empty. Then for every $1\le t\le R$, putting $M=p^{R-t+1}$, one has
\begin{equation}\label{eq:witt-product}
 C(Y)\equiv\prod_{1\le n<M}(1-x_nY^n)^{-1}\pmod{(p^t,Y^M)}.
\end{equation}
In particular,
\begin{equation}\label{eq:witt-modp}
 \sum_{k=0}^{p^R-1}a_{kp}Y^k\equiv\prod_{1\le n<p^R}(1-x_nY^n)^{-1}\pmod{Y^{p^R}}
\end{equation}
in $\Fp[[Y]]$.
\end{theorem}
\begin{proof}
Apply the big Witt factorization to the actual series $C$:
\[
 C(Y)=\prod_{n\ge1}(1-x_nY^n)^{-1}.
\]
Its ghost components are the coefficients of $L=\theta C/C$. By Theorem~\ref{thm:fixed-point}, $L^{(R-1)}\equiv L\pmod{p^R}$. At $n=p^sm$, the ghost identity becomes
\begin{equation}\label{eq:witt-layered-ghost}
 \ell_{p^sm}\equiv\sum_{j=0}^s p^j\sum_{d\mid m}d\,x_{p^jd}^{p^{s-j}m/d}\pmod{p^R}.
\end{equation}
All terms with $j<s$ involve indices strictly smaller than $p^sm$. After subtracting them, the remainder is divisible by $p^s$. Division by $p^s$ leaves, modulo $p^{R-s}$,
\[
 \sum_{d\mid m}d\,x_{p^sd}^{m/d}.
\]
All terms with $d<m$ are already known, and the remaining term is $mx_{p^sm}$. Since $p\nmid m$, multiplication by $m$ is invertible modulo $p^{R-s}$, which gives \eqref{eq:witt-recursion}. Notice that the stored precision of every previously computed coordinate is sufficient: a term from layer $j$ is multiplied by $p^j$ and only needs $x_{p^jd}$ modulo $p^{R-j}$.

Fix now $t\le R$ and put $M=p^{R-t+1}$. If $n<M$, then $\vp(n)\le R-t$, so the recursion determines $x_n$ modulo at least $p^t$. Thus the truncated product in \eqref{eq:witt-product} is computable modulo $p^t$. It has the same ghost components as $C$ modulo $p^R$ through the required range; alternatively, apply Lemma~\ref{lem:witt-rigidity} to the ratio of $C$ and the corresponding full Witt product. This proves \eqref{eq:witt-product}. Taking $t=1$ gives \eqref{eq:witt-modp}.
\end{proof}

\begin{proof}[Proof of Theorem~\ref{thm:recovery}]
Lemma~\ref{lem:fixed-operators} and Theorem~\ref{thm:fixed-point} construct the $1$-Lipschitz map $\mathcal T_p$ and give
\[
 L^{(R-1)}\equiv\frac{\theta C}{C}\pmod{p^R}
\]
for every $R\ge1$. Theorem~\ref{thm:witt-recovery} then reconstructs
\[
 C\pmod{(p^t,Y^{p^{R-t+1}})}
\]
for every $1\le t\le R$. Given $k$ and a desired coefficient precision $p^t$, choose $R$ with $k<p^{R-t+1}$ to recover $u_{kp}\pmod{p^t}$.

Finally, if $p\nmid n$, the exact recurrence
\[
 nu_n=\sum_{p^i\le n}u_{n-p^i}
\]
may be read modulo $p^t$, because $n$ is invertible modulo $p^t$. Induction computes every $u_n\pmod{p^t}$ once the multiples of $p$ are supplied by the $p$-section algorithm. Taking $t=1$ gives the entire reduced coefficient sequence.
\end{proof}

\begin{corollary}[Closed recovery through $p^3$]\label{cor:p3-recovery}
The explicit series $L_p^{[3]}$ of Appendix~\ref{app:third-order}, Proposition~\ref{prop:third-order}, satisfies
\[
 L_p^{[3]}\equiv L^{(2)}\equiv\frac{\theta C}{C}\pmod{p^3}.
\]
Consequently Theorem~\ref{thm:witt-recovery} with $R=3$ recovers every $a_{kp}$ for $0\le k<p^3$. In particular this includes all odd $k$ and the endpoint $k=p^3-1$, whereas the version of record of Avitabile and Mattarei \cite[Proposition~2]{AM} treats the even indices $k<p^3-1$. The original preprint stated the smaller range $k<p^2-1$; the published paper strengthens that proposition to $k<p^3-1$.
\end{corollary}

\begin{corollary}[Second-order recovery]\label{cor:p2-recovery}
The first fixed-point correction gives
\[
 L^{(1)}\equiv\frac{\theta C}{C}\pmod{p^2},
\]
and Theorem~\ref{thm:witt-recovery} with $R=2$ determines all $a_{kp}$ with $0\le k<p^2$.
\end{corollary}

\begin{remark}[No finite $p$-adic recovery barrier]
The third-order normal-ordering formulas are useful closed expressions, but they are not required in order to proceed to fourth and higher precision. The exact fixed-point equation \eqref{eq:fixed-point} packages all higher normal-ordering corrections into the integral $1$-Lipschitz map $\mathcal T_p$. Each iteration adds one correct $p$-adic digit to $\theta C/C$, and each additional digit extends the mod-$p$ reconstruction of $C$ from degrees below $p^r$ to degrees below $p^{r+1}$. Thus the earlier finite block and finite precision boundaries are features of explicit truncation, not intrinsic obstructions.
\end{remark}

\begin{remark}[Iterated $p$-sections and Cartier precision descent]\label{rem:cartier-descent}
For completeness, the fixed-point data also descend to any prescribed higher Cartier section, but with a depth-dependent precision cost.  Let
\[
 C_k(Y)=\Lambda_0^k\AH(Y)=\sum_{n\ge0}u_{p^kn}Y^n,
 \qquad L_k=\frac{\theta C_k}{C_k},
\]
where $\Lambda_0(\sum \xi_nY^n)=\sum \xi_{pn}Y^n$.  We use the standard projection formula
\[
 \Lambda_0\bigl(A(Y^p)B(Y)\bigr)=A(Y)\,\Lambda_0B(Y).
\]
Put
\[
 \Sigma_k(Y)=\frac{C_k(Y)}{C_{k+1}(Y^p)},
 \qquad \Gamma_k=\frac{\theta\Sigma_k}{\Sigma_k}.
\]
Applying the projection formula to $C_k(Y)=\Sigma_k(Y)C_{k+1}(Y^p)$ gives
\[
 C_{k+1}(Y)=(\Lambda_0\Sigma_k)(Y)C_{k+1}(Y),
\]
so, since $C_{k+1}$ is a unit, $\Lambda_0\Sigma_k=1$.  Taking logarithmic derivatives gives
\[
 L_k(Y)=\Gamma_k(Y)+pL_{k+1}(Y^p).
\]
Writing $\Sigma_k=1+\sum_{n\ge1}s_nY^n$ and $\Gamma_k=\sum_{n\ge1}\gamma_nY^n$, the condition $s_{pn}=0$ and the identity $\theta\Sigma_k=\Gamma_k\Sigma_k$ give a triangular recursion: for $p\nmid n$, $\gamma_n=[Y^n]L_k$ and one divides only by the unit $n$ to obtain $s_n$; for $p\mid n$, the equation $s_n=0$ determines $\gamma_n$ without division by $p$.  The final passage from $L_k-\Gamma_k$ to $L_{k+1}$ divides once by $p$.

Consequently, for every fixed $k\ge1$ and $m\ge1$, knowledge of $L_1\pmod{p^{m+k-1}}$ determines $L_k\pmod{p^m}$.  This is a finite-depth statement: the starting precision grows linearly with $k$, so it does not provide bounded-precision control uniformly as $k\to\infty$.
\end{remark}

\section{Consequences and remarks}\label{sec:consequences}
Let $\mathfrak s(p,k)$ denote the signed Stirling number of the first kind, so that
\[
 (Z)_p=\sum_{k=0}^p \mathfrak s(p,k)Z^k.
\]

\begin{corollary}[Artin--Hasse--Stirling--Bernoulli identity]
For $1<k<p$,
\[
 W_k=\frac{\mathfrak s(p,k)}p=\frac{B_{p-k}}k\qquad\text{in }\Fp.
\]
\end{corollary}
\begin{proof}
For $1<k<p$, the coefficient of $Z^k$ in $\Phi_p(Z)$ is $\mathfrak s(p,k)/p$. Proposition~\ref{prop:PhiBern} and Theorem~\ref{thm:global} give the result.
\end{proof}

The global logarithmic derivative does not by itself determine $G$ uniquely in characteristic $p$, because differentiation annihilates arbitrary series in $X^p$. Below degree $p$, however, ordinary integration is unambiguous.

\begin{corollary}\label{cor:explicit-G}
In $\Fp[X]/(X^p)$,
\[
 G(X)=\exp\!\left(w_pX+\sum_{n=1}^{p-2}\frac{B_n}{n^2}X^{p-n}\right).
\]
Thus every $a_{kp}$ with $0\le k<p$ is given explicitly by a complete Bell polynomial in $w_p$ and the divided Bernoulli residues.
\end{corollary}
\begin{proof}
Modulo $X^p$, Proposition~\ref{prop:global-log} reduces to \eqref{eq:old-logform}. Divide by $X$ and integrate term by term, using $G(0)=1$ and $(p-n)^{-1}=-n^{-1}$ in $\Fp$.
\end{proof}

\begin{remark}[Regular primes]
For $1\le j\le(p-3)/2$, Proposition~\ref{prop:PhiBern} gives
\[
 [X^{p-2j}]\Phi_p(X)=-\frac{B_{2j}}{2j}.
\]
Hence $p$ is regular if and only if all these coefficients are nonzero. Equivalently,
\[
 \ord_p s(p,p-2j)=1\qquad\left(1\le j\le\frac{p-3}{2}\right).
\]
This reformulation does not by itself resolve the infinitude of regular primes, but it places all irregular indices of a fixed prime in the first-order defect of the falling factorial.
\end{remark}

\begin{remark}[No block obstruction]
The apparent obstruction beyond the second block comes from expanding the factor $\AH(Y^p)^{1/p^2}$ separately. Lemma~\ref{lem:exact-section} instead keeps the complete factor
\[
 \mathcal B(Y)=\exp\!\left(\sum_{i\ge0}\frac{Y^{p^i}}{p^{i+1}}\right),
\]
whose conjugation replaces $p\theta$ by the integral operator $p\theta-q(Y)$. Lemma~\ref{lem:operator-frob} then controls its $p$th power on the whole of $\Zp[[Y]]$. Thus the denominators appearing in successive truncations are artifacts of the blockwise expansion rather than genuine obstructions to the weighted convolution. Theorem~\ref{thm:digital} goes further: coefficient extraction from the global identity itself is finite and explicit in every degree, so no residual blockwise dependence remains.
\end{remark}

\appendix
\section{Closed third-order expansion of the fixed-point map}\label{app:third-order}
This appendix records the normal-ordering calculation underlying the closed specialization through $p^3$.  It is independent of the all-orders Witt reconstruction in Section~\ref{sec:recovery}.  We first record the required operator expansions. In the next lemmas, $E$ is any $\mathbf Z_p$-linear operator on $\Zp[[Y]]$ satisfying
\[
 [E,M_g]=M_{\theta g}\qquad(g\in\Zp[[Y]]),
\]
while $f(Y)\in\Zp[[Y]]$, $f_j=\theta^jf$, and $\mu=M_f$ denotes multiplication by $f$. Both $E=\theta$ and $E=\theta+M_H$ satisfy this hypothesis.

\begin{lemma}[Polynomial perturbation through second order]\label{lem:poly-perturb}
For every $F(Z)\in\Zp[Z]$,
\begin{align}
F(\mu+pE)\equiv{}&F(\mu)+p\left(F'(\mu)E+\frac12F''(\mu)M_{f_1}\right)\label{eq:poly-perturb-a}\\
&+p^2\left(\frac12F''(\mu)E^2+\frac12F'''(\mu)M_{f_1}E
+\frac18F^{(4)}(\mu)M_{f_1^2}+\frac16F'''(\mu)M_{f_2}\right)
\pmod{p^3}.\label{eq:poly-perturb}
\end{align}
The displayed combinations are integral polynomial operators; in particular the formula is meaningful also when $p=3$.
\end{lemma}
\begin{proof}
It is enough to take $F(Z)=Z^j$ and temporarily replace $p$ by a central parameter $t$. The commutator hypothesis gives
\[
 E\mu^r=\mu^rE+r\mu^{r-1}M_{f_1}.
\]
Normal-ordering the words in $(\mu+tE)^j$ which contain one occurrence of $E$ gives
\[
 j\mu^{j-1}E+\binom j2\mu^{j-2}M_{f_1}.
\]
The sum of the words containing exactly two occurrences of $E$ is
\[
 \binom j2\mu^{j-2}E^2+3\binom j3\mu^{j-3}M_{f_1}E
 +3\binom j4\mu^{j-4}M_{f_1^2}+\binom j3\mu^{j-3}M_{f_2}.
\]
These are exactly the coefficients of $t$ and $t^2$ in \eqref{eq:poly-perturb-a}--\eqref{eq:poly-perturb}. Terms containing at least three occurrences of $tE$ are divisible by $t^3$. Setting $t=p$ proves the result. Integrality follows either from the binomial-coefficient form just displayed or directly from the divisibility of the relevant formal derivatives.
\end{proof}

For the $p$th power itself we need one further piece of normal ordering.
\begin{lemma}[Operator Frobenius congruence modulo $p^4$]\label{lem:frob-p4}
If $p\ge5$, then
\begin{align}
(\mu+pE)^p\equiv{}&\mu^p
+p^2\left(\mu^{p-1}E+\frac{p-1}{2}\mu^{p-2}M_{f_1}\right)\label{eq:frob-p4-a}\\
&+p^3\left(\frac{p-1}{2}\mu^{p-2}E^2
+\frac{(p-1)(p-2)}2\mu^{p-3}M_{f_1}E\right.\notag\\
&\hspace{13mm}\left.+\frac{(p-1)(p-2)(p-3)}8\mu^{p-4}M_{f_1^2}
+\frac{(p-1)(p-2)}6\mu^{p-3}M_{f_2}\right)
\pmod{p^4}.\label{eq:frob-p4}
\end{align}
For $p=3$ one instead has
\begin{align}
(\mu+3E)^3\equiv{}&\mu^3+9(\mu^2E+\mu M_{f_1}+M_{f_2})\label{eq:frob-3-a}\\
&+27(\mu E^2+M_{f_1}E+E^3)\pmod{81}.\label{eq:frob-3}
\end{align}
\end{lemma}
\begin{proof}
The terms containing no $E$, one $E$, and two $E$'s give the terms displayed in \eqref{eq:frob-p4-a}--\eqref{eq:frob-p4}, by the normal-ordering calculation in the proof of Lemma~\ref{lem:poly-perturb} with $F(Z)=Z^p$. Terms with at least four occurrences of $pE$ are automatically divisible by $p^4$.

It remains, for $p\ge5$, to consider the sum $S_3$ of the words containing exactly three occurrences of $E$. The coefficient of $t^3$ in $(\mu+tE)^p$ is
\begin{align*}
S_3={}&\frac16F'''(\mu)E^3
 +\frac14F^{(4)}(\mu)M_{f_1}E^2\\
&+\left(\frac16F^{(4)}(\mu)M_{f_2}
 +\frac18F^{(5)}(\mu)M_{f_1^2}\right)E\\
&+\frac1{24}F^{(4)}(\mu)M_{f_3}
 +\frac1{12}F^{(5)}(\mu)M_{f_1f_2}
 +\frac1{48}F^{(6)}(\mu)M_{f_1^3},
\end{align*}
where $F(Z)=Z^p$. Every coefficient in this expression is divisible by $p$ for $p\ge5$ (for $p=5$ the only potentially exceptional fifth-derivative coefficient is $15$). Hence $p^3S_3\equiv0\pmod{p^4}$.

For $p=3$ there are only eight terms to normal-order, and the exact contributions with one, two, and three occurrences of $3E$ give \eqref{eq:frob-3-a}--\eqref{eq:frob-3}.
\end{proof}

Return to $C(Y)=\sum_{n\ge0}u_{np}Y^n$ and put
\begin{equation}\label{eq:L0}
 f(Y)=-q(Y),\qquad L_{0,p}(Y)=\Phi_p(f(Y))-D_p(Y).
\end{equation}
For $p\ge5$ define the integral polynomials
\begin{equation}\label{eq:ABJK}
 \mathsf A_{p,1}(Z)=\frac{P_p'(Z)+1}{p},\qquad \mathsf A_{p,2}(Z)=\frac{P_p''(Z)}p,\qquad
 \mathsf A_{p,3}(Z)=\frac{P_p'''(Z)}p,\qquad \mathsf A_{p,4}(Z)=\frac{P_p^{(4)}(Z)}p.
\end{equation}
With $f_j=\theta^jf$ as above, set
\begin{equation}\label{eq:L1}
 L_{1,p}=\mathsf A_{p,1}(f)L_{0,p}+\frac12\mathsf A_{p,2}(f)f_1,
\end{equation}
\begin{align}
 L_{2,p}={}&\mathsf A_{p,1}(f)L_{1,p}+\frac12\mathsf A_{p,2}(f)(\theta L_{0,p}+L_{0,p}^2)
 +\frac12\mathsf A_{p,3}(f)f_1L_{0,p}\notag\\
&+\frac18\mathsf A_{p,4}(f)f_1^2+\frac16\mathsf A_{p,3}(f)f_2.\label{eq:L2}
\end{align}
For $p=3$, put
\begin{equation}\label{eq:A3}
 \mathsf A_{3,1}(Z)=\frac{P_3'(Z)+1}{3}=(Z-1)^2
\end{equation}
and define
\begin{equation}\label{eq:L1-3}
 L_{1,3}=\mathsf A_{3,1}(f)L_{0,3}+(f-1)f_1+f_2,
\end{equation}
\begin{equation}\label{eq:L2-3}
 L_{2,3}=\mathsf A_{3,1}(f)L_{1,3}+(f-1)(\theta L_{0,3}+L_{0,3}^2)
 +f_1L_{0,3}+\theta^2L_{0,3}+L_{0,3}^3.
\end{equation}
Finally set, for every odd prime $p$,
\begin{equation}\label{eq:Lp3-def}
 L_p^{[3]}(Y)=L_{0,p}(Y)+pL_{1,p}(Y)+p^2L_{2,p}(Y).
\end{equation}

\begin{proposition}[Third-order logarithmic derivative]\label{prop:third-order}
One has
\begin{equation}\label{eq:Lp3}
 \frac{\theta C(Y)}{C(Y)}\equiv L_p^{[3]}(Y)\pmod{p^3}
 \qquad\text{in }\Zp[[Y]].
\end{equation}
In particular, reducing \eqref{eq:Lp3} modulo $p^2$ gives
\begin{equation}\label{eq:Lp2}
 \frac{\theta C}{C}\equiv L_{0,p}+pL_{1,p}\pmod{p^2}.
\end{equation}
\end{proposition}
\begin{proof}
Write $L=\theta C/C$. We derive the expansion directly from the exact fixed-point equation \eqref{eq:fixed-point}; this also makes explicit where the $pL$ term created by conjugation enters the calculation. For an auxiliary $H\in\Zp[[Y]]$ put
\[
 E_H=\theta+M_H,\qquad T_H=M_f+pE_H.
\]
Thus $T_L=C^{-1}(p\theta+f)C$. For every $g\in\Zp[[Y]]$,
\[
 [E_H,M_g]=M_{\theta g},
\]
so Lemmas~\ref{lem:poly-perturb} and \ref{lem:frob-p4} apply directly with $E=E_H$. When the resulting operators are applied to $1$, the only additional evaluations are
\begin{equation}\label{eq:EH-on-one}
 E_H1=H,\qquad E_H^21=\theta H+H^2.
\end{equation}
This is precisely how the conjugation shift $+pH$ is accounted for; it is not to be absorbed into $f$. See Remark~\ref{rem:conjugation-shift} for the associated precision warning.

Assume first that $p\ge5$. Applying the two operator expansions to $T_H=M_f+pE_H$, dividing as in \eqref{eq:Q-V} and \eqref{eq:V-def}, and then evaluating at $1$ gives
\begin{align}
\mathcal T_p(H)\equiv{}&\mathsf A_{p,1}(f)H+\frac12\mathsf A_{p,2}(f)f_1\notag\\
&+p\left[\frac12\mathsf A_{p,2}(f)(\theta H+H^2)+\frac12\mathsf A_{p,3}(f)f_1H
+\frac18\mathsf A_{p,4}(f)f_1^2+\frac16\mathsf A_{p,3}(f)f_2\right]\pmod{p^2}.\label{eq:F-third}
\end{align}
Indeed, the coefficients in the first line are the sums
\[
 f^{p-1}+\Phi_p'(f)=\mathsf A_{p,1}(f),\qquad
 (p-1)f^{p-2}+\Phi_p''(f)=\mathsf A_{p,2}(f),
\]
and the analogous identities give $\mathsf A_{p,3}$ and $\mathsf A_{p,4}$ in the second line. Substituting $H=L$ in \eqref{eq:F-third} and using $L=L_{0,p}+p\mathcal T_p(L)$ yields
\begin{align}
L\equiv{}&L_{0,p}+p\left(\mathsf A_{p,1}(f)L+\frac12\mathsf A_{p,2}(f)f_1\right)\notag\\
&+p^2\left[\frac12\mathsf A_{p,2}(f)(\theta L+L^2)+\frac12\mathsf A_{p,3}(f)f_1L
+\frac18\mathsf A_{p,4}(f)f_1^2+\frac16\mathsf A_{p,3}(f)f_2\right]\pmod{p^3}.\label{eq:L-third}
\end{align}
Now insert
\[
 L=L_{0,p}+pL_{1,p}+p^2L_{2,p}.
\]
Comparison first modulo $p^2$ and then modulo $p^3$ gives exactly \eqref{eq:L1} and \eqref{eq:L2}.

For $p=3$, the same argument uses the exceptional Frobenius expansion \eqref{eq:frob-3-a}--\eqref{eq:frob-3} together with $\Phi_3(Z)=-Z^2+Z$. It gives, for arbitrary $H$,
\[
 \mathcal T_3(H)\equiv \mathsf A_{3,1}(f)H+(f-1)f_1+f_2
 +3\bigl((f-1)(\theta H+H^2)+f_1H+\theta^2H+H^3\bigr)\pmod9.
\]
Putting $H=L$ in the fixed-point equation, replacing $L$ by $L_{0,3}$ in the final bracket and by $L_{0,3}+3L_{1,3}$ where one further digit is required, gives \eqref{eq:L1-3} and \eqref{eq:L2-3}.
\end{proof}

\begin{remark}[On the conjugation shift]\label{rem:conjugation-shift}
Formula~\eqref{eq:Lp2} is sometimes easy to misread if one first conjugates $p\theta+f$ and then forgets that the perturbation operator is $E_L=\theta+M_L$, not the bare derivation $\theta$. The conjugated operator is
\[
 T_L=M_f+p(\theta+M_L),
\]
not $M_f+p\theta$. In the proof above the $+pL$ shift is retained in the perturbation operator $E_L=\theta+M_L$. Since $[E_L,M_f]=M_{f_1}$, the polynomial coefficients $\mathsf A_{p,1}(f),\mathsf A_{p,2}(f),\ldots$ continue to be evaluated at the original $f=-q$; the shift reappears through $E_L1=L$ and $E_L^21=\theta L+L^2$. Thus no replacement of $f$ by $f+pL$ is required. There is a second precision point in \eqref{eq:Lp2}: $L_{0,p}=\Phi_p(f)-D_p$ is an integral $p$-adic series and must itself be evaluated modulo $p^2$. Replacing $L_{0,p}$ by its reduction modulo $p$ before forming $L_{0,p}+pL_{1,p}$ changes the approximation modulo $p^2$ for the Artin--Hasse series; this is an error, not merely a loss of sharpness.
\end{remark}

\section*{Declaration of generative AI and AI-assisted technologies in the manuscript preparation process}
During the development and preparation of this work, the author used OpenAI's ChatGPT, including Codex, as an interactive research and writing aid to explore candidate approaches, assist in developing and checking algebraic arguments, support literature searches, and help with the drafting and revision of the manuscript. All outputs from these tools were critically reviewed and edited by the author. The author independently verified every mathematical argument, calculation, and reference appearing in the manuscript, made all final mathematical and editorial decisions, and takes full responsibility for the content of the article.

\end{document}